\documentclass[12pt]{amsart}
\usepackage{amssymb, amsmath, mathrsfs, verbatim, url, bbm}
\usepackage{marvosym}
\usepackage{wasysym}

\theoremstyle{plain}
\newtheorem{theorem}{Theorem}
\newtheorem{lemma}[theorem]{Lemma}
\newtheorem{proposition}[theorem]{Proposition}
\newtheorem{corollary}[theorem]{Corollary}

\theoremstyle{definition}
\newtheorem{definition}[theorem]{Definition}

\theoremstyle{remark}

\newcommand{\ran}{\textrm{ran}}

\newcommand{\Nst}{\mathbb{N}^\ast}
\newcommand{\CC}{\mathcal{C}}
\newcommand{\DD}{\mathcal{D}}
\newcommand{\EE}{\mathcal{E}}
\newcommand{\FF}{\mathcal{F}}

\newcommand{\LL}{\mathcal{L}}
\newcommand{\UU}{\mathcal{U}}
\newcommand{\VV}{\mathcal{V}}
\newcommand{\NNN}{\mathbb{N}}
\newcommand{\SSS}{\mathbb{S}}
\newcommand{\TTT}{\mathbb{T}}

\newcommand{\cccc}{\mathfrak{c}}

\begin{document}

\title[A non-CLP-compact product space]{A non-CLP-compact product space whose finite subproducts are CLP-compact}
\author{Andrea Medini}

\date{September 6, 2010}

\address{University of Wisconsin-Madison Mathematics Dept.}
\email{medini@math.wisc.edu}

\begin{abstract}
We construct a family of Hausdorff spaces such that every finite product of spaces in the family (possibly with repetitions) is CLP-compact, while the product of all spaces in the family is non-CLP-compact. Our example will yield a single Hausdorff space $X$ such that every finite power of $X$ is CLP-compact, while no infinite power of $X$ is CLP-compact. This answers a question of Stepr\={a}ns and \v{S}ostak.
\end{abstract}

\maketitle

For all undefined topological notions, see \cite{engelking}. Recall that a subset of a topological space is \emph{clopen} if it is closed and open.

\begin{definition}
A space $X$ is \emph{CLP-compact} if every cover of $X$ consisting of clopen sets has a finite subcover.
\end{definition}

Using methods of Stepr\={a}ns and \v{S}ostak (see Section 4 in \cite{steprans}), we will construct a family $\{X_i:i\in\omega\}$ of Hausdorff spaces such that $\prod_{i\in n}X_{p(i)}$ is CLP-compact whenever $n\in\omega$ and $p:n\longrightarrow\omega$, while the full product $P=\prod_{i\in\omega}X_i$ is not. In order to make sure that $P$ is non-CLP-compact, we will use an idea from \cite{comfort}; see also Example 8.28 in \cite{walker}. Our example answers the first half of Question 6.3 in \cite{steprans}; see also Question 7.1 in \cite{dikranjan}. For a positive result on (finite) products of CLP-compact spaces, see \cite{steprans2}.

In the last section, using ideas from \cite{frolik} and \cite{comfort2}, we will convert the example described above into a single Hausdorff space $X$ such that $X^n$ is CLP-compact for every $n\in\omega$, while $X^\kappa$ is non-CLP-compact for every infinite cardinal $\kappa$. This answers the second half of Question 6.3 in \cite{steprans} (see Corollary $\ref{secondhalf}$).

Let $\FF$ be a family of non-empty closed subsets of $\Nst$. Define $X(\FF)$ as the topological space with underlying set $\NNN\cup\FF$ and with the coarsest topology satisfying the following requirements.
\begin{itemize}
\item The singleton $\{n\}$ is open for every $n\in\NNN$.
\item For every $K\in\FF$, the set $\{K\}\cup A$ is open whenever $A\subseteq\NNN$ is such that $K\subseteq A^\ast$.
\end{itemize}
Our example will be obtained by setting $X_i=X(\FF_i)$ for every $i\in\omega$, where $\{\FF_i:i\in \omega\}$ is the family constructed in Theorem $\ref{construction}$. 

\begin{proposition}[Stepr\={a}ns and \v{S}ostak] Assume that $\FF$ consists of pairwise disjoint sets. Then $X(\FF)$ is a Hausdorff space.
\end{proposition}
\begin{proof}
Use the fact that disjoint closed sets in $\Nst$ can be separated by a clopen set.
\end{proof}

\section{Finite products}

We will use the following notion to ensure the CLP-compactness of finite products.
\begin{definition} 
Fix $n\in\omega$. A product $X=\prod_{i\in n}X_i$ is \emph{CLP-rectangular} if for every clopen set $U\subseteq X$ and every $x\in U$ there exists a clopen rectangle $R\subseteq X$ such that $x\in R\subseteq U$.
\end{definition}
For the proof of the following proposition, see Proposition 2.4 and Proposition 2.5 in \cite{steprans}; see also Theorem 3.4 in \cite{dikranjan}.

\begin{proposition}[Stepr\={a}ns and \v{S}ostak]\label{clprectangular} Fix $n\in\omega$. Assume that $X_i$ is CLP-compact for every $i\in n$. Then $\prod_{i\in n}X_i$ is CLP-compact if and only if it is CLP-rectangular.
\end{proposition}

From now on, we will always assume that $\FF_i$ is a collection of non-empty pairwise disjoint closed subsets of $\Nst$ for every $i\in\omega$. Given $p\in {}^{<\omega}\omega$, we will denote by $n(p)\in\omega$ the domain of $p$. We will use the notation
$$
X_p=\prod_{i\in n(p)}X(\FF_{p(i)})
$$
for finite products, where repetitions of factors are allowed. Also, if $i\in n(p)$, we will denote by $X_{p-i}$ the subproduct $\prod_{j\in n(p)\setminus\{i\}}X(\FF_{p(j)})$.

The following definitions isolate the multidimensional versions of `finiteness' and `cofiniteness' that we need. For every $p\in  {}^{<\omega}\omega$ and $N\in\omega$, we will denote the union of the `initial stripes of height $N$' as
$$
S^N_p=\bigcup_{i\in n(p)}\{x\in X_p: x_i\in N\}\subseteq X_p.
$$
Also define
$$
T^N_p=X_p\setminus S^N_p=\prod_{i\in n(p)}(X(\FF_{p(i)})\setminus N).
$$
\begin{proposition}\label{fincofimplies}
Assume that for every $p\in  {}^{<\omega}\omega$ and every clopen set $U\subseteq X_p$, either $U\subseteq S^N_p$ or $T^N_p\subseteq U$ for some $N\in\omega$. Then $X_p$ is CLP-compact for every $p\in  {}^{<\omega}\omega$.
\end{proposition}
\begin{proof}
We will use induction on $n(p)$. The case $n(p)=1$ is obvious. So assume that $X_p$ is CLP-compact for every $p\in  {}^{n}\omega$ and let $p\in  {}^{n+1}\omega$. By Proposition $\ref{clprectangular}$, it is enough to prove that $X_p$ is CLP-rectangular. So let $U\subseteq X_p$ be a clopen set and fix $N\in\omega$ such that $U\subseteq S^N_p$ or $T^N_p\subseteq U$. We will show that for every $x\in U$ there exists a clopen rectangle $R\subseteq X_p$ such that $x\in R\subseteq U$.

First we will assume that $x\in U$ has at least one coordinate in $N$, say coordinate $i\in n(p)$. It is easy to check that the cross-section
$$
V=\{y\in X_{p-i}:y\cup\{(i,x_i)\}\in U\}
$$
is clopen in $X_{p-i}$. Observe that $X_{p-i}$ is homeomorphic to $X_{p'}$ for some $p'\in {}^{n}\omega$. Therefore, by the inductive hypothesis and Proposition $\ref{clprectangular}$, there exists a clopen rectangle $Q\subseteq X_{p-i}$ such that $\pi_{p-i}(x)\in Q\subseteq V$, where $\pi_{p-i}:X_p\longrightarrow X_{p-i}$ is the natural projection. It is clear that the desired clopen rectangle is $R=\{y\in X_p :\pi_{p-i}(y)\in Q\textrm{ and }y_i=x_i\}$.

On the other hand, if $x\in U$ has no coordinate in $N$ then the case $U\subseteq S^N_p$ is impossible. Therefore $T^N_p\subseteq U$, so that the desired clopen rectangle is $R=T^N_p$ itself.
\end{proof}

We will also need the following definitions. Let
$$
\SSS^N_p=S^N_p\cap (\NNN^{n(p)})=\bigcup_{i\in n(p)}\{x\in \NNN^{n(p)}: x_i\in N\},
$$
$$
\TTT^N_p=T^N_p\cap (\NNN^{n(p)})=\NNN^{n(p)}\setminus\SSS^N_p=(\NNN\setminus N)^{n(p)}.
$$
The next two lemmas show that, in order to achieve what is required by Proposition $\ref{fincofimplies}$, we can just look at the trace of clopen sets on $\NNN^{n(p)}$.

\begin{lemma}\label{finimpliesfin}
Fix $p\in  {}^{<\omega}\omega$. Assume that $U\subseteq X_p$ is a clopen set such that $U\cap (\NNN^{n(p)})\subseteq \SSS^N_p$. Then $U\subseteq S^N_p$.
\end{lemma}
\begin{proof}
Assume, in order to get a contradiction, that $x\in U\setminus S^N_p$. For all $i\in n(p)$ such that $x_i\in\NNN$, let $N_i=\{x_i\}\subseteq\NNN\setminus N$. For all $i\in n(p)$ such that $x_i\in \FF_{p(i)}$, let $N_i=\{x_i\}\cup A_i$ be a neighborhood of $x_i$ in $X(\FF_{p(i)})$ such that $A_i\subseteq\NNN\setminus N$. Since $U$ is open, by shrinking each $N_i$ if necessary, we can make sure that $\prod_{i\in n(p)}N_i\subseteq U$. This is a contradiction, because $\varnothing\neq (\prod_{i\in n(p)}N_i)\cap (\NNN^{n(p)})\subseteq \TTT^N_p$.
\end{proof}

\noindent Similarly, one can prove the following.

\begin{lemma}\label{cofimpliescof}
Fix $p\in  {}^{<\omega}\omega$. Assume that $U\subseteq X_p$ is a clopen set such that $\TTT^N_p\subseteq U\cap (\NNN^{n(p)})$. Then $T^N_p\subseteq U$.
\end{lemma}

Fix $p\in  {}^{<\omega}\omega$. We will say that a subset $D$ of $\NNN^{n(p)}$ is \emph{diagonal} if $D\nsubseteq \SSS^N_p$ and $\TTT^N_p\nsubseteq D$ for all $N\in\omega$ and the restriction $\pi_i\upharpoonright D$ of the natural projection $\pi_i:\NNN^{n(p)}\longrightarrow \NNN$ is injective for every $i\in n(p)$.

Given $p\in  {}^{<\omega}\omega$, we will say that a pair $(D,E)$ is \emph{$p$-diagonal} if $D$ and $E$ are both diagonal subsets of $\NNN^{n(p)}$. A pair $(D,E)$ is \emph{diagonal} if it is $p$-diagonal for some $p$ as above. If $(D,E)$ is such a pair, consider the following statement.
\begin{enumerate}
\item[$\textrm{\Coffeecup}(D,E)$] There exist $K_0,\ldots,K_{n(p)-1}$, with $K_i\in \FF_{p(i)}$ for every $i\in n(p)$, such that $D\cap (A_{0}\times\cdots \times A_{n(p)-1})$ and $E\cap (A_{0}\times\cdots \times A_{n(p)-1})$ are both non-empty whenever $A_{0},\ldots ,A_{n(p)-1}\subseteq \NNN$ satisfy $K_i\subseteq A_i^\ast$ for every $i\in n(p)$.
\end{enumerate}

\begin{proposition}\label{coffeeimplies} Fix $p\in  {}^{<\omega}\omega$. Assume that the family $\{\FF_i:i\in \omega\}$ is such that condition $\textrm{\Coffeecup}(D,E)$ holds for every $p$-diagonal pair $(D,E)$. If $U\subseteq X_p$ is a clopen set, then there exists $N\in\omega$ such that either $U\cap (\NNN^{n(p)})\subseteq \SSS^N_p$ or $\TTT^N_p\subseteq U\cap (\NNN^{n(p)})$.
\end{proposition}
\begin{proof}
Assume, in order to get a contradiction, that $U\cap (\NNN^{n(p)})\nsubseteq \SSS^N_p$ and $\TTT^N_p\nsubseteq U\cap (\NNN^{n(p)})$ for every $N\in\omega$. Then it is possible to construct (in $\omega$ steps) diagonal subsets $D$ and $E$ of $\NNN^{n(p)}$ such that $D\subseteq U$ and $E\subseteq X_p\setminus U$.

Now let $K_{0},\ldots ,K_{n(p)-1}$ be as given by condition $\textrm{\Coffeecup}(D,E)$. Define $x\in X_p$ by setting $x_i=K_i$ for every $i\in n(p)$. It is easy to see that $x\in\overline{U}\cap (\overline{X_p\setminus U})$, which contradicts the fact that $U$ is clopen.
\end{proof}

\begin{theorem}\label{finitesubproducts}
Assume that the family $\{\FF_i:i\in \omega\}$ is such that condition $\textrm{\Coffeecup}(D,E)$ holds for every diagonal pair $(D,E)$.
Then $X_p$ is CLP-compact for every $p\in  {}^{<\omega}\omega$.
\end{theorem}
\begin{proof}
We will show that the hypothesis of Proposition $\ref{fincofimplies}$ holds. This follows from Proposition $\ref{coffeeimplies}$, Lemma $\ref{finimpliesfin}$ and Lemma $\ref{cofimpliescof}$.
\end{proof}

\section{The full product}

In this section we will show how to ensure that $P=\prod_{i\in\omega}X(\FF_i)$ is non-CLP-compact. Consider the following condition. Recall that a family $\LL$ is \emph{linked} if $K\cap L\neq\varnothing$ whenever $K,L\in\LL$.

\begin{enumerate}
\item[$\textrm{\Bicycle}$] For every $I\in [\omega]^\omega$, the family $\LL=\{x_i:i\in I\}$ is not linked for any $x\in\prod_{i\in I}\FF_i$.
\end{enumerate}

\begin{theorem}\label{fullproduct}
Assume that the family $\{\FF_i:i\in\omega\}$ is such that condition $\textrm{\Bicycle}$ holds. Then $P=\prod_{i\in\omega}X(\FF_i)$ can be written as the disjoint union of infinitely many of its non-empty clopen subsets.
\end{theorem}
\begin{proof}
For each $n\in\omega$, define
$$
U_n=\{x\in P:x_i=n\textrm{ whenever }0\leq i\leq n\}.
$$
It is easy to check that each $U_n$ is open (actually, clopen), non-empty, and that $U_i\cap U_j=\varnothing$ whenever $i\neq j$. Therefore we just need to show that $V=P\setminus\bigcup_{n\in\omega}U_n$ is open.

So fix $x\in V$ and consider $I=\{i\in\omega: x_i\notin\NNN\}$. First assume that $I$ is finite. If there exist $i,j\notin I$, say with $i<j$, such that $x_i\neq x_j$ then $\{y\in P:y_i=x_i\textrm{ and }y_j=x_j\}\setminus\bigcup_{n\in j}U_n$ is an open neighborhood of $x$ which is  contained in $V$. So assume that $x_i=x_j$ whenever $i,j\notin I$. Since $x\in V$, we must have $I\neq\varnothing$. So fix $i\in I$ and $j\notin I$, say with $i<j$ (the other case is similar). Let $N_i=\{x_i\}\cup A_i$ be a neighborhood of $x_i$ such that $x_j\notin A_i$. Then $\{y\in P:y_i\in N_i\textrm{ and }y_j=x_j\}\setminus\bigcup_{n\in j}U_n$ is an open neighborhood of $x$ which is  contained in $V$.

Finally, assume that $I$ is infinite. An application of condition $\textrm{\Bicycle}$ yields $i,j\in I$, say with $i<j$, such that $x_i\cap x_j=\varnothing$. But disjoint closed sets in $\Nst$ can be separated by a clopen set, therefore we can find disjoint clopen neighborhoods $N_i$ and $N_j$ of $x_i$ and $x_j$ respectively. Then $\{y\in P:y_i\in N_i\textrm{ and }y_j\in N_j\}\setminus\bigcup_{n\in j}U_n$ is an open neighborhood of $x$ which is  contained in $V$.
\end{proof}

\section{The construction}

The next theorem guarantees the existence of our example: finite products will be CLP-compact by Theorem $\ref{finitesubproducts}$, while the full product will be non-CLP-compact by Theorem $\ref{fullproduct}$.

\begin{theorem}\label{construction}
There exists a family $\{\FF_i:i\in \omega\}$ satisfying the following requirements.
\begin{itemize}
\item Each $\FF_i$ consists of pairwise disjoint subsets of $\Nst$ of finite size.
\item The condition $\textrm{\Coffeecup}(D,E)$ holds for every diagonal pair $(D,E)$.
\item The condition $\textrm{\Bicycle}$ holds.
\end{itemize}
\end{theorem}
\begin{proof}
Enumerate as $\{(D_\eta,E_\eta):\eta\in\cccc\}$ all diagonal pairs, where $D_\eta$ and $E_\eta$ are both diagonal subsets of $\NNN^{n(p)}$ for some $p=p(\eta)\in  {}^{<\omega}\omega$ with domain $n(p)=n(\eta)\in\omega$.

We will construct $\{\FF_i:i\in \omega\}$ by transfinite recursion in $\cccc$ steps: in the end we will set $\FF_i=\bigcup_{\xi\in\cccc} \FF_i^\xi$ for every $i\in\omega$. Start with $\FF_i^0=\varnothing$ for each~$i$. By induction, we will make sure that the following requirements are satisfied.
\begin{enumerate}
\item $\FF_i^\eta\subseteq\FF_i^\mu$ whenever $\eta\leq\mu\in\cccc$.
\item\label{small} $|\bigcup_{i\in\omega}\bigcup\FF^\eta_i|<2^{\cccc}$ for every $\eta\in\cccc$.
\item\label{killbound} The condition $\textrm{\Coffeecup}(D_\eta,E_\eta)$, where $(D_\eta,E_\eta)$ is a $p$-diagonal pair, is satisfied at stage $\xi=\eta+1$: that is, the witness $K_i$ is already in $\FF_{p(i)}^{\eta+1}$ for each $i\in n(p)$.
\end{enumerate}

At a limit stage $\xi$, just let $\FF_i^\xi=\bigcup_{\eta\in\xi}\FF_i^\eta$ for every $i\in\omega$.

At a successor stage $\xi=\eta+1$, assume that $\FF_i^\eta$ is given for each $i$. Let $p=p(\eta)$. First, define $W=\bigcup_{i\in \omega}\bigcup\FF^\eta_i$ and observe that $|W|<2^\cccc$ by $(\ref{small})$. Set $\tau_i=\pi_i\upharpoonright D_\eta$ for every $i\in n(p)$. Since each $\tau_i$ is injective, it makes sense to consider the induced function $\tau_i^\ast : D_\eta^\ast\longrightarrow \NNN^\ast$. Recall that the explicit definition is given by
$$
\tau_i^\ast(\UU)=\{S\subseteq\NNN: \tau_i^{-1}[S]\in\UU\}.
$$
It is easy to check that each $\tau_i^\ast$ is injective. Therefore, since $|D_\eta^\ast|=2^\cccc$, it is possible to choose
$$
\UU^\eta\in D_\eta^\ast\setminus ((\tau_{0}^\ast)^{-1}[W]\cup\cdots\cup(\tau_{n(p)-1}^\ast)^{-1}[W]).
$$
Let $\UU^\eta_i=\tau_i^\ast(\UU^\eta)$ for every $i\in n(p)$.

\noindent Now, define $Z=W\cup\{\UU^\eta_i:i\in n(p)\}$. Set $\sigma_i=\pi_i\upharpoonright E_\eta$ for every $i\in n(p)$. As above, it is possible to choose
$$
\VV^\eta\in E_\eta^\ast\setminus ((\sigma_{0}^\ast)^{-1}[Z]\cup\cdots\cup (\sigma_{n(p)-1}^\ast)^{-1}[Z]).
$$
Let $\VV^\eta_i=\sigma_i^\ast(\VV^\eta)$ for every $i\in n(p)$.

\noindent We conclude the successor stage by setting
$$
\FF_k^{\eta+1}=\FF_k^\eta\cup\{\{\UU^\eta_i:i\in p^{-1}(k)\}\cup\{\VV^\eta_i:i\in p^{-1}(k)\}\}
$$
for every $k\in\ran (p)$ and $\FF_{k}^{\eta+1}=\FF_{k}^\eta$ for every $k\in\omega\setminus\ran (p)$.

Next, we will verify that condition $\textrm{\Bicycle}$ holds. Assume, in order to get a contradiction, that $I\in [\omega]^\omega$ and $x\in\prod_{i\in I}\FF_i$ are such that $\LL=\{x_i:i\in I\}$ is linked. Observe that the only possible equalities among points of $\Nst$ produced in our construction are those in the form $\UU^\eta_i=\UU^\eta_j$ or $\VV^\eta_i=\VV^\eta_j$ for some $i,j\in n(\eta)$. Therefore each element of $\LL$ must have been added to some $\FF_k$ at the same stage $\xi=\eta+1$. Since $I$ is infinite, we can pick $k\in I\setminus\ran (p(\eta))$. It is clear from the construction that $x_k\in\LL$ cannot have been added to $\FF_k$ at stage $\xi=\eta+1$.

Finally, we will verify that $(\ref{killbound})$ holds. For every $i\in n(p)$, set 
$$
K_i=\{\UU^\eta_j:j\in p^{-1}(p(i))\}\cup\{\VV^\eta_j:j\in p^{-1}(p(i))\}\in\FF_{p(i)}^{\eta+1}.
$$
Suppose $A_{0},\ldots,A_{n(p)-1}\subseteq\NNN$ are such that $K_i\subseteq A_i^\ast$ for every $i\in n(p)$. In particular $A_i\in\UU^\eta_i$ for every $i\in n(p)$. By the definition of the induced functions, we have $\tau_{i}^{-1}[A_i]\in \UU^\eta$ for every $i\in n(p)$. Therefore
$$
D_\eta\cap (A_{0}\times\cdots \times A_{n(p)-1})=D_\eta\cap\tau_{0}^{-1}[A_{0}]\cap \cdots\cap \tau_{n(p)-1}^{-1}[A_{n(p)-1}]
$$
is non-empty. By the same argument, using $\VV^\eta$, one can show that $E_\eta\cap (A_{0}\times\cdots \times A_{n(p)-1})$ is non-empty.
\end{proof}

\section{Arbitrarily large products}

The main idea behind the next theorem is due to Frol\'{\i}k (see the `Proof of $\textrm{B}$ (using $\textrm{B}'$)' in \cite{frolik}). To show that $X^\omega$ can be written as the disjoint union of infinitely many of its non-empty clopen subsets, we will proceed as in the proof of Theorem 9.10 in \cite{comfort2}.

\begin{theorem}\label{singlespace}
There exists a Hausdorff space $X$ such that $X^n$ is CLP-compact for every $n\in\omega$, while $X^\omega$ can be written as the disjoint union of infinitely many of its non-empty clopen subsets.
\end{theorem}
\begin{proof}
Let $\{\FF_i:i\in\omega\}$ be the family given by Theorem $\ref{construction}$ and set $X_i=X(\FF_i)$ for every $i\in\omega$. It follows from Theorem $\ref{finitesubproducts}$ and Theorem $\ref{fullproduct}$ that $\{X_i:i\in\omega\}$ is a collection of Hausdorff spaces such that $X_p=\prod_{i\in n(p)}X_{p(i)}$ is CLP-compact for every $p\in  {}^{<\omega}\omega$, while $\prod_{i\in\omega}X_i$ can be written as the disjoint union of infinitely many of its non-empty clopen subsets.

Define $X$ as the topological space with underlying set the disjoint union $\{0\}\oplus X_0\oplus X_1\oplus\cdots$ and with the coarsest topology satisfying the following requirements.
\begin{itemize}
\item Whenever $U$ is an open subset of $X_i$ for some $i\in\omega$, the set $U$ is also open in $X$.
\item The tail $\{0\}\cup\bigcup_{i\leq j<\omega}X_j$ is open in $X$ for every $i\in\omega$.
\end{itemize}
It is easy to check that $X$ is Hausdorff.

We will prove that $X^n$ is CLP-compact by induction on $n$. The case $n=1$ is obvious. So assume that $X^n$ is CLP-compact and consider a cover $\CC$ of $X^{n+1}$ consisting of clopen sets. Since
$$
S=\bigcup_{i\in n+1}\{x\in X^{n+1}:x_i=0\}
$$
is a finite union of subspaces of $X^{n+1}$ that are homeomorphic to $X^n$, there exist $\DD\in[\CC]^{<\omega}$ such that $S\subseteq \bigcup\DD$. It follows that there exists $N\in\omega$ such that $X^{n+1}\setminus (X_0\oplus\cdots\oplus X_{N-1})^{n+1}\subseteq\bigcup\DD$. But 
$$
T=(X_0\oplus\cdots\oplus X_{N-1})^{n+1}
$$
is homeomorphic to a finite union of spaces of the form $X_p$ for some $p\in {}^{n+1}\omega$, hence it is CLP-compact. Therefore there exists $\EE\in[\CC]^{<\omega}$ such that $T\subseteq\bigcup\EE$. Hence $\DD\cup\EE$ is the desired finite subcover of $\CC$.

Finally, we will show that $X^\omega$ can be written as the disjoint union of infinitely many of its non-empty clopen subsets by constructing a continuous surjection $f:X^\omega\longrightarrow\prod_{i\in\omega}X_i$. Since every $X_i$ is clopen in $X$, we can get a continuous surjection $f_i:X\longrightarrow X_i$ by letting $f_i$ be the identity on $X_i$ and constant on $X\setminus X_i$. Now simply let $f=\prod_{i\in\omega}f_i$ (that is, for every $x\in X^\omega$, define $y=f(x)$ by setting $y_i=f_i(x_i)$ for every $i\in\omega$).
\end{proof}

\begin{corollary}\label{secondhalf}
For every infinite cardinal $\kappa$, there exists a collection $\{X_\xi:\xi\in\kappa\}$ of Hausdorff spaces such that $\prod_{\xi\in F}X_\xi$ is CLP-compact for every $F\in [\kappa]^{<\omega}$, while $\prod_{\xi\in\kappa}X_\xi$ is non-CLP-compact.
\end{corollary}
\begin{proof}
Let $X_\xi=X$ for every $\xi\in\kappa$, where $X$ is the space given by Theorem $\ref{singlespace}$.
\end{proof}

\noindent\textbf{Acknowledgement.} The author thanks Wistar Comfort for valuable bibliographical informations.

\end{document}